\documentclass[11pt]{amsart}

\theoremstyle{plain}
\newtheorem{theorem}                {Theorem}      [section]
\newtheorem{proposition}  [theorem]  {Proposition}
\newtheorem{corollary}    [theorem]  {Corollary}
\newtheorem{lemma}        [theorem]  {Lemma}
\newtheorem{remark}       [theorem]  {Remark}
\newtheorem*{theorem1}                {Theorem \ref{thm_red}}
\newtheorem*{theorem2}                {Theorem \ref{thm:int}}
\newtheorem*{theorem3}                {Corollary \ref{theorem3}}
\newtheorem*{theorem4}                {Theorem \ref{theorem4}}
\theoremstyle{definition}

\newtheorem{definition}   [theorem]  {Definition}

\DeclareMathOperator{\trace}{trace}

 \DeclareMathOperator{\id}{I}

\DeclareMathOperator{\Span}{span}
 
\DeclareMathOperator{\cst}{constant}
 
 \DeclareMathOperator{\im}{Im}

\numberwithin{equation}{section}

\begin{document}

\title[Surfaces with parallel mean curvature vector]
{Surfaces with parallel mean curvature in Sasakian space forms}

\author{Dorel~Fetcu}
\author{Harold~Rosenberg}

\address{Department of Mathematics and Informatics\\
Gh. Asachi Technical University of Iasi\\
Bd. Carol I no. 11 \\
700506 Iasi, Romania} \email{dfetcu@math.tuiasi.ro}

\curraddr{Department of Mathematics, Federal University of Bahia, Av.
Adhemar de Barros s/n, 40170--110 Salvador, BA, Brazil}

\address{Instituto Nacional de Matem\'atica Pura e Aplicada\\ Estrada Dona Castorina 110\\ 22460--320 Rio de
Janeiro, RJ, Brazil} \email{rosen@impa.br}

\thanks{The first author was supported by a fellowship, BJT 373672/2013--6, offered by CNPq, Brazil.}

\begin{abstract} We study the global geometry of surfaces in Sasakian space forms whose mean curvature vector is parallel in the normal bundle (these include the Riemannian Heisenberg space of dimension $2n+1$). We prove a codimension reduction theorem.  We introduce two holomorphic quadratic differentials on anti-invariant such surfaces and use them to obtain classification theorems.
\end{abstract}

\date{}

\subjclass[2010]{53A10, 53C42, 53C25}

\keywords{surfaces with parallel mean curvature vector field, Sasakian
space forms}

\maketitle

\section{Introduction}

There is considerable research on the geometry of constant mean curvature surface (cmc surfaces) in $3$-dimensional manifolds. In this paper we will consider surfaces in manifolds of dimension greater than $3$, whose mean curvature vector is parallel in the normal bundle (pmc surfaces).

The theory of pmc surfaces was developed by S.-T.~Yau \cite{Y}, D. Ferus \cite{DF} and J. Erbacher \cite{E}. Recently, H. Alencar, M. do Carmo and R. Tribuzy obtained structure theorems for pmc surfaces in $M^n\times\mathbb{R}$, where $M^n$ is a simply connected space form \cite{AdCT}.  The present authors developed the theory in complex and cosymplectic space forms \cite{F,FR}. M.~J.~Ferreira and R. Tribuzy continued this study in symmetric spaces \cite{FT}.

We now study pmc surfaces in simply connected Sasakian space forms $N^{2n+1}(c)$ of constant $\varphi$-sectional curvature $c$ (we will explain what this means in the next section). Remark that $N^{2n+1}(1)$ is isometric to the unit sphere $\mathbb{S}^{2n+1}$ and $N^{2n+1}(-3)$ is isometric to Heisenberg space. We recommend the book of D. E. Blair [2].

One of our main results is the following theorem. 

\begin{theorem1} Let $\Sigma^2$ be an isometrically immersed non-minimal pmc surface in a Sasakian space form $N^{2n+1}(c)$ with constant $\varphi$-sectional curvature $c\neq 1$. Then one of the following holds$:$
\begin{enumerate}
\item $\Sigma^2$ is an integral pseudo-umbilical surface and $n\geq 3$$;$ or

\item $\Sigma^2$ is not pseudo-umbilical and lies in a Sasakian space form $N^{11}(c)$.
\end{enumerate}
\end{theorem1}

We also study anti-invariant pmc surfaces in $N^{2n+1}(c)$ with $c\neq 1$ and mean curvature vector $\vec H\neq 0$. We introduce two holomorphic quadratic differentials on such surfaces and use them to understand their geometry. In particular, we classify all integral complete non-minimal pmc surfaces with non-negative Gaussian curvature when $n=3$ (dimension $7$). The theorem is (we will explain the terminology in Section \ref{s4}):

\begin{theorem2} Let $\Sigma^2$ be an isometrically immersed integral complete non-minimal surface in a Sasakian space form $N^7(c)$ with parallel mean curvature vector field $\vec H$ and with non-negative Gaussian curvature $K$. If $Q_1^{(2,0)}$ vanishes on the surface or, equivalently, if $\Sigma^2$ is pseudo-umbilical, then $|\vec H|^2\geq-(c+3)/4$ and one of the following holds$:$
\begin{enumerate}

\item $\Sigma^2$ is a cmc totally umbilical surface in a space form $M^3((c+3)/4)$, with constant sectional curvature $(c+3)/4$ and of dimension $3$, immersed in $N^7(c)$ as an integral submanifold$;$ or

\item $\Sigma^2$ is flat and it is the standard product $\gamma_1\times\gamma_2$, where $\gamma_1:\mathbb{R}\rightarrow N^7(c)$ is a Legendre helix of osculating order $4$ in $N^7(c)$ with curvatures 
$$
\kappa_1=\sqrt{a^2+|\vec H|^2},\quad \kappa_2=\frac{a\sqrt{1+|\vec H|^2}}{\sqrt{a^2+|\vec H|^2}},\quad\textnormal{and}\quad\kappa_3=\frac{|\vec H|\sqrt{1+|\vec H|^2}}{\sqrt{a^2+|\vec H|^2}},
$$
and $\gamma_2:\mathbb{R}\rightarrow N^7(c)$ is a Legendre circle in $N^7(c)$ with curvature $\kappa=\sqrt{a^2+|\vec H|^2}$, where $a^2=(c+3)/8+|\vec H|^2/2$.
\end{enumerate}
\end{theorem2}

\begin{theorem3} An integral pmc $2$-sphere in $N^7(c)$, with $\vec H\neq 0$, is a round sphere in a space form $M^3((c+3)/4)$.
\end{theorem3}

The paper is organized as follows. In Section \ref{sintro} we define Sasakian manifolds and state their curvature equations. We give the four models of the simply connected Sasakian space forms $N^{2n+1}(c)$ obtained by S. Tanno \cite{ST}.

In Section \ref{s3} we develop the study of the geometry of pmc surfaces in the space forms $N^{2n+1}(c)$. We then prove the reduction of codimension result, Theorem \ref{thm_red}.

In Section \ref{s4} we discuss anti-invariant pmc surfaces in $N^{2n+1}(c)$. We introduce two quadratic differentials and prove that their $(2,0)$-parts are holomorphic quadratic differentials on such surfaces. We define Hopf cylinders and discuss their geometry. We then prove Theorem \ref{thm:int}.

Finally, in Section \ref{s4}, we pursue the study of anti-invariant pmc surfaces and we prove the following non-existence theorem.

\begin{theorem4} There are no anti-invariant complete non-minimal non-pseudo-umbilical pmc surfaces with non-negative Gaussian curvature in a Sasakian space form $N^{2n+1}(c)$, with $c\neq 1$. In particular, there are no anti-invariant non-minimal pmc $2$-spheres in $N^{2n+1}(c)$.
\end{theorem4}

\noindent {\bf Acknowledgments.} The first author would like to thank the Department of Mathematics of the Federal University of Bahia in Salvador for providing a very stimulative work environment during the preparation of this paper.

\section{Preliminaries}\label{sintro}

A \textit{contact metric structure} on an odd-dimensional manifold
$N^{2n+1}$ is given by $(\varphi,\xi,\eta,\langle,\rangle)$, where $\varphi$ is
a tensor field of type $(1,1)$ on $N$, $\xi$ is a vector field,
$\eta$ is its dual $1$-form and $\langle,\rangle$ is a Riemannian metric such that
$$
\varphi^{2}U=-U+\langle U,\xi\rangle\xi,\quad
\langle\varphi U,\varphi V\rangle=\langle U,V\rangle-\eta(U)\eta(V),
$$
and 
$$
d\eta(U,V)=\langle U,\varphi V\rangle,
$$
for all tangent vector fields $U$ and $V$. Such a structure is
called {\it normal} if
$$
\mathcal N_{\varphi}(U,V)+2d\eta(U,V)\xi=0,
$$
where
$$
\mathcal N_{\varphi}(U,V)=[\varphi U,\varphi V]-\varphi \lbrack \varphi
U,V]-\varphi \lbrack U,\varphi V]+\varphi^{2}[U,V],
$$
is the Nijenhuis tensor field of $\varphi$.

\begin{definition} A normal contact metric manifold $(N,\varphi,\xi,\eta,g)$ is called a
\textit{Sasakian manifold}.
\end{definition}

Equivalently, a contact metric manifold is a Sasakian manifold if and only if
$$
(\nabla^N_{U}\varphi)V=\langle U,V\rangle\xi-\eta(V)U,
$$
for all tangent vector fields $U$ and $V$, where $\nabla^N$ is the Levi-Civita connection. It can be easily shown that on a Sasakian manifold we have
$\nabla^N_U\xi=-\varphi U$ (see ~\cite{B}). 

A submanifold $M$ of a Sasakian manifold $N^{2n+1}$ is called {\it anti-invariant} when $\varphi(TM)\subset NM$, where $NM$ is the normal bundle of $M$, and {\it integral} if $\eta(X)=0$ for all vector fields $X$ tangent to $M$. The dimension $m$ of an anti-invariant submanifold satisfies $m\leq n+1$ and that of an integral submanifold $m\leq n$. We also note that any integral submanifold of a Sasakian manifold is anti-invariant. An integral curve is called a {\it Legendre curve}.

Now, let $(N,\varphi,\xi,\eta,g)$ be a Sasakian manifold. The sectional
curvature of the $2$-plane generated by $U$ and $\varphi U$, where $U$
is a unit vector orthogonal to $\xi$, is called the
\textit{$\varphi$-sectional curvature} determined by $U$. A Sasakian
manifold with constant $\varphi$-sectional curvature $c$ is called a
\textit{Sasakian space form} and it is denoted by $N^{2n+1}(c)$.

The curvature tensor field of a Sasakian space form $N^{2n+1}(c)$ is given by
\begin{align}\label{eq:curv}
R^N(U,V)W=&\frac{c+3}{4}\Big\{\langle W,V\rangle U-\langle W,U\rangle V\Big\}+\frac{c-1}{4}\Big\{\eta(W)\eta(U)V\\\nonumber 
&-\eta(W)\eta(V)U+\langle W,U\rangle\eta(V)\xi-\langle W,V\rangle\eta(U)\xi\\\nonumber
&+\langle W,\varphi V\rangle\varphi U-\langle W,\varphi U\rangle\varphi
V+2\langle U,\varphi V\rangle\varphi W\Big\}.
\end{align}

In \cite{O}, it is proved that a Sasakian manifold is locally symetric if and only if it is of constant sectional curvature $1$. However, all Sasakian space forms are {\it locally $\varphi$-symmetric} spaces, i.e.,
$$
\varphi^2(\nabla^N_UR^N)(X,Y)Z=0,
$$
for all tangent vector fields $U$, $X$, $Y$ and $Z$ orthogonal to $\xi$. In order to characterize locally $\phi$-symmetric spaces, a very useful tool proved to be the affine connection $\bar\nabla$ introduced by M. Okumura, defined on a Sasakian manifold by
$$
\bar\nabla_UV=\nabla^N_UV+\mathcal{T}_UV,
$$
where $\mathcal{T}_UV=\langle U,\varphi V\rangle\xi-\eta(U)\varphi V+\eta(V)\varphi U$ (see \cite{O}). The torsion $\bar T(U,V)=2\mathcal{T}_UV$ of this connection does not vanish, but it is parallel with respect to $\bar\nabla$, and T.~Takahashi showed that a Sasakian manifold is locally $\varphi$-symmetric if and only if the curvature $\bar R$ of $\bar\nabla$ is parallel: $\bar\nabla\bar R=0$. Here $\bar R$ is given by
\begin{align*}
\bar R(U,V)W=&R^N(U,V)W+\eta(W)\{\eta(U)V-\eta(V)U\}+\langle\varphi V,W\rangle\varphi U\\\nonumber &-\langle\varphi U,W\rangle\varphi V+2\langle\varphi U,V\rangle\varphi W+\{\langle U,W\rangle\eta(V)-\langle V,W\rangle\eta(U)\}\xi.
\end{align*}
This is equivalent to
\begin{align}\label{symmetric}
(\nabla^N_UR^N)(X,Y)Z=&-\mathcal{T}_UR^N(X,Y)Z+R^N(\mathcal{T}_UX,Y)Z+R^N(X,\mathcal{T}_UY)Z\\\nonumber &+R^N(X,Y)\mathcal{T}_UZ,
\end{align}
for all tangent vector fields $U$, $X$, $Y$ and $Z$ (see \cite{T}). It is easy to verify that equation \eqref{symmetric} holds on any Sasakian space form.

Complete simply connected Sasakian space forms $N^{2n+1}(c)$ were classified by S.~Tanno in \cite{ST}, as follows:

\begin{itemize}

\item if $c>-3$,
then either $N^{2n+1}(c)$ is isometric to the unit sphere $\mathbb{S}^{2n+1}$
endowed with its canonical Sasakian structure, or $N(c)$ is isometric to
$\mathbb{S}^{2n+1}$ endowed with a deformed Sasakian structure, described in the following.

Let $\mathbb{S}^{2n+1}=\{z\in\mathbb{C}^{n+1}: \vert z\vert=1\}$ be
the unit $2n+1$-dimensional sphere endowed with its standard metric
field $\langle,\rangle_{0}$. Consider the following structure tensor fields on
$\mathbb{S}^{2n+1}$: $\xi_{0}=-\mathcal{J}z$ for each $z\in
\mathbb{S}^{2n+1}$, where $\mathcal{J}$ is the usual complex
structure on $\mathbb{C}^{n+1}$, and $\varphi_{0}=s\circ \mathcal{J}$, where
$s:T_{z}\mathbb{C}^{n+1}\to T_{z}\mathbb{S}^{2n+1}$ denotes the
orthogonal projection. Equipped with these tensors,
$\mathbb{S}^{2n+1}$ becomes a Sasakian space form with
$\varphi_{0}$-sectional curvature equal to $1$.

Next, consider a deformed structure on $\mathbb{S}^{2n+1}$, given by
$$
\eta=a\eta_{0},\quad\xi=\frac{1}{a}\xi_{0},\quad
\varphi=\varphi_{0},\quad \langle U,V\rangle=a\langle U,V\rangle_{0}+a(a-1)\eta_{0}(U)\eta_{0}(V),
$$
where $a$ is a positive constant. Then 
$(\mathbb{S}^{2n+1},\varphi,\xi,\eta,\langle,\rangle)$ is a Sasakian space form
with constant $\varphi$-sectional curvature $c=\frac{4}{a}-3>-3$.

\item if $c=-3$, then $N^{2n+1}(c)$ is isometric to the generalized Heisenberg group $\mathbb{R}^{2n+1}$ endowed with the following Sasakian structure. On $\mathbb{R}^{2n+1}$, with coordinates $(x^1,\ldots,x^n,y^1,\ldots,y^n,z)$, consider the vector field $\xi=2(\partial/\partial z)$, its dual $1$-form $\eta=(1/2)(dz-\sum_{i=1}^ny^idx^i)$, the tensor field $\varphi$ given by the matrix
$$
\left(\begin{array}{ccc}0&\delta_{ij}&0\\-\delta_{ij}&0&0\\0&y^j&0\end{array}\right),
$$
and the Riemannian metric $\langle,\rangle=(1/4)\sum_{i=1}^n((dx^i)^2+(dy^i)^2)+\eta\otimes\eta$. Then $(\mathbb{R}^{2n+1},\varphi,\xi,\eta,\langle,\rangle)$ becomes a Sasakian space form with constant $\varphi$-sectional curvature $c=-3$.

\item if $c<-3$, then $N^{2n+1}(c)$ is isometric to $B^{2n}\times\mathbb{R}$, where $B^{2n}$ is the unit ball in $\mathbb{C}^n$, with the Sasakian structure given by the $1$-form $\eta=\pi^{\ast}\omega+dt$ and the Riemannian metric $\langle,\rangle=\pi^{\ast}G+\eta\otimes\eta$, where $(J,G)$ is a K\"ahler structure with constant sectional holomorphic curvature $k<0$ and exact fundamental $2$-form $\Omega=d\omega$, $\pi:B^{2n}\times\mathbb{R}\rightarrow B^{2n}$ is the canonical projection, and $t$ is the coordinate on $\mathbb{R}$. Endowed with this structure, $B^{2n}\times\mathbb{R}$ becomes a Sasakian space form with constant $\varphi$-sectional curvature $c=k-3$.
\end{itemize}

Throughout our paper, we shall work in the above described spaces, that we will simply call Sasakian space forms.

\section{A reduction of codimension theorem}\label{s3}

Let $\Sigma^2$ be an isometrically immersed surface in a Sasakian space form $N^{2n+1}(c)$,  endowed with the Sasakian
structure $(\varphi,\xi,\eta,\langle,\rangle)$ and with constant $\varphi$-sectional curvature $c$. The second fundamental form $\sigma$ of the surface is then defined by the equation of Gauss
$$
\nabla^N_XY=\nabla_XY+\sigma(X,Y),
$$
while the shape operator $A$ and the normal connection $\nabla^{\perp}$ are given by the equation of Weingarten
$$
\nabla^N_XV=-A_VX+\nabla^{\perp}_XV,
$$
for any vector fields $X$ and $Y$ tangent to the surface and any vector field $V$ normal to $\Sigma^2$, where $\nabla^N$ and $\nabla$ are the Levi-Civita connections on $N^{2n+1}(c)$ and $\Sigma^2$, respectively. The mean curvature vector field $\vec H$ of $\Sigma$ is given by $\vec H=(1/2)\trace\sigma$.

We also have the Gauss equation
\begin{align}\label{Gauss}
\langle R(X,Y)Z,W\rangle=&\langle R^N(X,Y)Z,W\rangle+\langle\sigma(Y,Z),\sigma(X,W)\rangle\\\nonumber&-\langle\sigma(X,Z),\sigma(Y,W)\rangle,
\end{align}
the Codazzi equation
\begin{equation}\label{Codazzi}
(R^N(X,Y)Z)^{\perp}=(\nabla^{\perp}_X\sigma)(Y,Z)-(\nabla^{\perp}_Y\sigma)(X,Z),
\end{equation}
and the equation of Ricci
\begin{equation}\label{Ricci}
\langle R^{\perp}(X,Y)U,V\rangle=\langle [A_U,A_V]X,Y\rangle+\langle R^N(X,Y)U,V\rangle,
\end{equation}
for any vector fields $X$, $Y$, $Z$ and $W$ tangent to $\Sigma^2$ and any normal vector fields $U$ and $V$, where $R^N$, $R$ and $R^{\perp}$ are the curvature tensors corresponding to $\nabla^N$, $\nabla$, and $\nabla^{\perp}$, respectively. 

\begin{definition} If the mean curvature vector $\vec H$ of the surface $\Sigma^2$ is
parallel in the normal bundle, i.e.,
$\nabla^{\perp}\vec H=0$, then $\Sigma^2$ is called a \textit{pmc surface}.
\end{definition}

The pmc surfaces in Euclidean sphere were studied by S.-T. Yau in \cite{Y}. Henceforth we shall assume that our ambient space $N^{2n+1}(c)$ has constant $\varphi$-sectional curvature $c\neq 1$.

Let us now consider a non-minimal pmc surface $\Sigma^2$ isometrically immersed in $N^{2n+1}(c)$. Since the map $p\in\Sigma^2\rightarrow(A_{\vec H}-|\vec H|^2\id)(p)$
is analytic, it follows that if $\vec H$ is an umbilical direction, then this either holds on the whole surface or only for a closed set without interior points. In the last case
$\vec H$ is not an umbilical direction in an open dense set. We shall split our study in two cases as $\vec H$ is umbilical everywhere or it is not umbilical on an open dense set.

First, we need the following lemma.

\begin{lemma}\label{lemma_com} For any vector $V$ normal to $\Sigma^2$, which is also orthogonal to
$\varphi T\Sigma^2$ and to $\varphi \vec H$, we have $[A_{\vec H},A_V]=0$, i.e., $A_{\vec H}$
commutes with $A_V$.
\end{lemma}

\begin{proof} The conclusion follows from the Ricci equation \eqref{Ricci},
since
\begin{align*}
\langle R^N(X,Y)\vec H,V\rangle &=\frac{c-1}{4}\{\langle X,\varphi \vec H\rangle\langle\varphi Y,V\rangle-\langle Y,\varphi\vec H\rangle\langle\varphi X,V\rangle+2\langle X, \varphi Y\rangle\langle\varphi\vec H,V\rangle\}\\&=0
\end{align*}
and $R^{\perp}(X,Y)\vec H=0$.
\end{proof}

\begin{corollary}\label{lemma_split} Either $\vec H$ is an
umbilical direction or there exists a basis that diagonalizes
simultaneously $A_{\vec H}$ and $A_V$, for all normal vectors $V$ satisfying $V\perp \varphi T\Sigma^2$
and $V\perp \varphi \vec H$.
\end{corollary}

\subsection{$\vec H$ is an umbilical direction} In this case the surface is pseudo-umbilical, i.e., $A_{\vec H}=|H|^2\id$, and since $\vec H$ is also parallel, we have
\begin{align*}
R^N(X,Y)\vec H&=\nabla^N_X\nabla^N_Y\vec H-\nabla^N_Y\nabla^N_X\vec H-\nabla^N_{[X,Y]}\vec H=-|\vec H|^2(\nabla^N_XY-\nabla^N_YX-[X,Y])\\&=0,
\end{align*}
for any tangent vector fields $X$ and $Y$.

We shall prove that $\Sigma^2$ is an integral surface and $n\geq 3$. First, we have

\begin{lemma}\label{l1umb} If $\Sigma^2$ is a non-minimal pmc pseudo-umbilical surface in a space form $N^{2n+1}(c)$, with $c\neq 1$, then the following are equivalent:
\begin{enumerate}
\item[({\it i})] $\Sigma^2$ is an integral surface $;$

\item[({\it ii})] $\Sigma^2$ is anti-invariant$;$

\item[({\it iii})] $\vec H\perp\xi$$;$

\item[({\it iv})] $\varphi\vec H\perp T\Sigma^2$.

\end{enumerate}
\end{lemma}

\begin{proof} First, assume that $\Sigma^2$ is integral. Then, from the definition of the Sasakian structure, it follows that $\langle X,\varphi Y\rangle=d\eta(X,Y)=0$, for any tangent vector fields $X$ and $Y$, which means that our surface is anti-invariant. Moreover, since $\eta(X)=0$ for any tangent vector field $X$, we also have 
$$
0=Y(\eta(X))=\langle\nabla^N_YX,\xi\rangle+\langle X,\nabla^N_Y\xi\rangle=\langle\sigma(X,Y),\xi\rangle-\langle X,\varphi Y\rangle=\langle\sigma(X,Y),\xi\rangle,
$$
which implies that $\eta(\vec H)=0$.

Finally, from $R^N(X,Y)\vec H=0$ and the fact that $c\neq 1$, we obtain that 
\begin{equation}\label{eq_gen}
\eta(\vec H)\eta(X)Y-\eta(\vec H)\eta(Y)X-\langle\varphi\vec H,Y\rangle\varphi X+\langle\varphi\vec H,X\rangle\varphi Y+2\langle\varphi Y,X\rangle\varphi\vec H=0, 
\end{equation}
and then, taking the inner product with $\varphi Y$ and since $\Sigma^2$ is integral, that $\langle\varphi\vec H,X\rangle=0$, i.e., $\varphi\vec H$ is orthogonal to $T\Sigma^2$. 

We have just proved that $(i)$ implies $(ii)$, $(iii)$ and $(iv)$.

Next, we will show that $(ii)$ implies $(i)$. Let us choose an orthonormal frame field $\{E_1,E_2\}$ on $\Sigma^2$ such that $E_2\perp\xi$, i.e., $\eta(E_2)=0$. Then, from \eqref{eq_gen}, we get
\begin{equation}\label{ast}
\eta(\vec H)\eta(E_1)E_2-\langle\varphi\vec H,E_2\rangle\varphi E_1+\langle\varphi\vec H,E_1\rangle\varphi E_2+2\langle\varphi E_2,E_1\rangle\varphi\vec H=0.
\end{equation}
Taking the inner product with $E_2$ and with $\varphi E_2$, respectively, one obtains
\begin{equation}\label{eq:R2}
\eta(\vec H)\eta(E_1)+3\langle\varphi E_2,E_1\rangle\langle\varphi\vec H,E_2\rangle=0
\end{equation}
and
\begin{equation}\label{eq:Rsuppl}
\langle\varphi\vec H,E_1\rangle=0.
\end{equation}

Now, if $\Sigma^2$ is anti-invariant, from \eqref{eq:R2}, it follows that $\eta(\vec H)\eta(E_1)=0$. But, since $\Sigma^2$ is pseudo-umbilical, we also have
$$
E_1(\eta(\vec H))=\langle\nabla^N_{E_1}\vec H,\xi\rangle+\langle\vec H,\nabla^N_X\xi\rangle=-|\vec H|^2\eta(E_1)-\langle\vec H,\varphi E_1\rangle=-|\vec H|^2\eta(E_1),
$$
which shows that $\eta(E_1)=0$. This means that $\Sigma^2$ is an integral surface.

If $(iii)$ holds on $\Sigma^2$, then we have $0=E_1(\eta(\vec H))=-|\vec H|^2\eta(E_1)$,
which again implies that $\Sigma^2$ is integral.

Finally, assume that $\varphi\vec H\perp T\Sigma^2$. Then, from \eqref{eq:R2}, we obtain that $\eta(\vec H)\eta(E_1)=0$, which, as we have seen before, means that $\Sigma^2$ is an integral surface.
\end{proof}

\begin{proposition}\label{p_umb} Let $\Sigma^2$ be a non-minimal pmc surface in a Sasakian space form $N^{2n+1}(c)$, with $c\neq 1$. If the mean curvature vector field $\vec H$ is an umbilical direction everywhere, then $n\geq 3$ and $\Sigma^2$ is an integral surface.
\end{proposition}

\begin{proof} Let us again consider the orthonormal frame field $\{E_1,E_2\}$ on $\Sigma^2$ used in the proof of Lemma \ref{l1umb}. We shall first prove some equations that will be used later on.

Taking the inner product of \eqref{ast} with $\varphi \vec H$ and $\varphi E_1$, respectively, and using \eqref{eq:Rsuppl}, one obtains
\begin{equation}\label{eq:R1}
\eta(E_1)\eta(\vec H)\langle\varphi \vec H,E_2\rangle+\langle\varphi E_2,E_1\rangle\langle\varphi \vec H,\varphi \vec H\rangle=0,
\end{equation}
and
\begin{equation}\label{eq:R3}
3\eta(E_1)\eta(\vec H)\langle\varphi E_2,E_1\rangle+\langle\varphi \vec H,E_2\rangle\langle\varphi E_1,\varphi E_1\rangle=0.
\end{equation}

We shall now assume that $\langle\varphi E_2,E_1\rangle\neq 0$ and we shall prove that this leads to a contradiction. From Lemma \eqref{l1umb} and equation \eqref{eq:Rsuppl} it follows that $\langle\varphi \vec H,E_2\rangle\neq 0$. Then, from \eqref{eq:R2} and \eqref{eq:R1} we have
\begin{equation}\label{eq:R1.1}
\langle\varphi \vec H,\varphi \vec H\rangle=|\vec H|^2-(\eta(\vec H))^2=3\langle\varphi \vec H,E_2\rangle^2
\end{equation}
and, from \eqref{eq:R2} and \eqref{eq:R3}, one obtains
\begin{equation}\label{eq:R1.2}
\langle\varphi E_1,\varphi E_1\rangle=1-(\eta(E_1))^2=9\langle\varphi E_2,E_1\rangle^2.
\end{equation}

Now, from equation \eqref{eq:R1.1}, it follows that $E_1(|\varphi\vec H|^2)=3E_1(\langle\varphi\vec H,E_2\rangle^2)$, i.e.,
$$
\langle\varphi \vec H, \varphi\nabla^N_{E_1}\vec H-\eta(H)E_1\rangle=3\langle\varphi \vec H,E_2\rangle(\langle\varphi\nabla^N_{E_1}\vec H-\eta(H)E_1,E_2\rangle+\langle\varphi \vec H,\nabla^N_{E_1}E_2\rangle).
$$
Using \eqref{eq:R2} and \eqref{eq:Rsuppl}, since $\nabla^N_{E_1}\vec H=-|\vec H|^2E_1$ and 
$$
\langle\varphi \vec H,\nabla_{E_1}E_2\rangle=\langle\nabla_{E_1}E_2,E_1\rangle\langle\varphi \vec H,E_1\rangle=0,
$$
we get
\begin{equation}\label{eq:12}
\langle\varphi \vec H,\sigma(E_1,E_2)\rangle=-2|\vec H|^2\langle\varphi E_2,E_1\rangle.
\end{equation}

In the same way, from equation \eqref{eq:R1.1}, we obtain 
\begin{equation}\label{eq:8}
\langle\varphi \vec H,\sigma(E_1,E_1)\rangle=-\langle\varphi \vec H,\sigma(E_2,E_2)\rangle=-\frac{2}{3}\eta(\vec H).
\end{equation}
From \eqref{eq:Rsuppl}, it follows that $\langle\varphi \vec H,\nabla^N_{E_1}E_1\rangle=\eta(\vec H)$, which implies, using \eqref{eq:8}, that
\begin{equation}\label{eq:9}
\langle\varphi \vec H,\nabla_{E_1}E_1\rangle=\frac{5}{3}\eta(\vec H).
\end{equation}

Next, from \eqref{eq:R2} and \eqref{eq:9}, one obtains that $\eta(\nabla_{E_1}E_2)=5\langle\varphi E_2,E_1\rangle$. Then, since $\eta(E_2)=0$, we have $0=\langle\nabla^N_{E_1}E_2,\xi\rangle+\langle E_2,\nabla^N_{E_1}\xi\rangle=\langle\nabla^N_{E_1}E_2,\xi\rangle-\langle E_2,\varphi E_1\rangle$, which gives
\begin{equation}\label{eq:nonumber}
\eta(\sigma(E_1,E_2))=-6\langle\varphi E_2,E_1\rangle.
\end{equation}

Finally, taking the inner product of \eqref{ast} with $\sigma(E_1,E_2)$ and using \eqref{eq:12}, we get
\begin{equation}\label{eq:nn}
\langle\varphi\vec H,E_2\rangle\langle\varphi E_1,\sigma(E_1,E_2)\rangle=-4|H|^2\langle\varphi E_2,E_1\rangle^2.
\end{equation}

As we have seen in Section \ref{sintro}, the Sasakian space form $N^{2n+1}(c)$ is a $\varphi$-symmetric space, which means that
equation \eqref{symmetric} holds on $N^{2n+1}(c)$. This implies that
\begin{align*}
(\nabla^N_{E_2}R^N)(E_1,E_2)\vec H=&-\mathcal{T}_{E_2}R^N(E_1,E_2)\vec H+R^N(\mathcal{T}_{E_2}E_1,E_2)\vec H+R^N(E_1,\mathcal{T}_{E_2}E_2)\vec H\\\nonumber &+R^N(E_1,E_2)\mathcal{T}_{E_2}\vec H.
\end{align*}

After a long but straightforward computation, using equations \eqref{eq:curv}, \eqref{ast} and \eqref{eq:12}--\eqref{eq:nn}, and the fact that $R^N(E_1,E_2)\vec H=0$, the last equation becomes
\begin{align*}
\eta(\vec H)\eta(\sigma(E_2,E_2))E_1-(2\eta(E_1)\langle\varphi\vec H,E_2\rangle-9\eta(\vec H)\langle\varphi E_2,E_1\rangle)E_2&\\-\frac{1}{3}\eta(\vec H)\varphi E_1+5|\vec H|^2\langle\varphi E_2,E_1\rangle\varphi E_2-2(\eta(E_1)+\langle\varphi E_2,\sigma(E_1,E_2)\rangle&\\-\langle\sigma(E_2,E_2),\varphi E_1\rangle)\varphi\vec H+\langle\varphi\vec H,E_2\rangle\varphi\sigma(E_1,E_2)-\eta(E_1)\eta(\vec H)\sigma(E_2,E_2)&=0
\end{align*}
and, by taking the inner product with $E_1$ and using \eqref{eq:Rsuppl} and \eqref{eq:nn}, we obtain that
\begin{equation}\label{eq:R_par}
\eta(\vec H)\eta(\sigma(E_2,E_2))+9|\vec H|^2\langle\varphi E_2,E_1\rangle^2=0.
\end{equation}

Next, from equations \eqref{eq:R2}, \eqref{eq:R1.1} and \eqref{eq:R1.2}, it follows that
$$
3|\vec H|^2\langle\varphi e_2,e_1\rangle^2=(1-6\langle\varphi e_2,e_1\rangle^2)(\eta(\vec H))^2.
$$
Then, from \eqref{eq:R_par}, since $\eta(\vec H)\neq 0$ by Lemma \ref{l1umb}, we get that 
$$
\eta(\sigma(E_2,E_2))=-3\eta(\vec H)(1-6\langle\varphi E_2,E_1\rangle^2),
$$ 
or, equivalently, 
$$
\eta(\sigma(E_1,E_1))=\eta(\vec H)(5-18\langle\varphi e_2,e_1\rangle^2).
$$
It is easy to see that $\eta(E_2)=0$ implies $\eta(\nabla_{E_1}E_1)=0$, which means that 
\begin{equation}\label{eq:eta}
\eta(\nabla^N_{E_1}E_1)=\eta(\vec H)(5-18\langle\varphi E_2,E_1\rangle^2).
\end{equation}

Next, from \eqref{eq:R1.2}, we see that $9\langle\varphi E_2,E_1\rangle E_1(\langle\varphi E_2,E_1\rangle)+\eta(E_1)\eta(\nabla^N_{E_1}E_1)=0$. Then, from \eqref{eq:eta} and \eqref{eq:R2}, it follows
\begin{equation}\label{eq:e2}
E_1(\langle\varphi E_2,E_1\rangle)=\Big(\frac{5}{3}-6\langle\varphi E_2,E_1\rangle^2\Big)\langle\varphi \vec H,E_2\rangle.
\end{equation}

From \eqref{eq:R2} we get $E_1(\eta(\vec H)\eta(E_1)+3\langle\varphi E_2,E_1\rangle\langle\varphi\vec H,E_2\rangle)=0$,
that, since $\vec H$ is umbilical and parallel, and using in this order equations \eqref{eq:eta}, \eqref{eq:12}, \eqref{eq:e2}, \eqref{eq:R1.1} and \eqref{eq:R1.2}, leads to
$$
|\vec H|^2+(5-18\langle\varphi e_2,e_1\rangle^2)(\eta(\vec H))^2=0.
$$
Since $\Sigma^2$ is non-minimal and equation \eqref{eq:R1.2} implies that $\langle\varphi e_2,e_1\rangle^2<1/9$, we can see that the last equation is actually a contradiction. Therefore, $\langle\varphi E_2,E_1\rangle=0$, i.e., $\Sigma^2$ is anti-invariant and then integral, by Lemma \ref{l1umb}. 

Again using Lemma \ref{l1umb}, it can be easily seen that $\{\varphi E_1,\varphi E_2, H,\varphi H,\xi\}$ are linearly independent vector fields in the normal bundle of the surface, which means that $n\geq 3$.
\end{proof}

\subsection{$\vec H$ is not an umbilical direction} As we have seen, in this case $\vec H$ is not umbilical on an open dense set. We shall work on this set and then we shall extend our results to the whole surface by continuity. It will turn out that the codimension of such a surface $\Sigma^2$ in $N^{2n+1}(c)$ can be reduced to $9$ and then that the surface lies in an $11$-dimensional Sasakian space form $N^{11}(c)$.

\begin{proposition}\label{lemma_parallel} If $\vec H$ is not an umbilical direction then there exists a parallel subbundle of the normal bundle that
contains the image of the second fundamental form $\sigma$ and has dimension less than or equal to $9$.
\end{proposition}

\begin{proof} Let $L$ be a subbundle of the normal bundle, defined by
$$
L=\Span\{\im\sigma\cup(\varphi(\im\sigma))^{\perp}\cup(\varphi(T\Sigma^2))^{\perp}\cup\xi^{\perp}\},
$$
where
$(\varphi(T\Sigma^2))^{\perp}=\{(\varphi X)^{\perp}:X\ \textnormal{tangent to}\
\Sigma^2\}$, $(\varphi(\im\sigma))^{\perp}=\{(\varphi\sigma(X,Y))^{\perp}:X,Y\ \textnormal{tangent to}\
\Sigma^2\}$ and $\xi^{\perp}$ is the normal component of $\xi$ along the surface. We shall prove that $L$ is parallel.

If $V$ is a normal vector field orthogonal to $L$, it is easy to verify that so it is $\varphi V$.

Next, we shall prove that for a normal vector field $V$ orthogonal to $L$ also $\nabla^{\perp} V$ is orthogonal to $L$, which means that $L$ is parallel.

First, since $V$ is orthogonal to $L$, we have
$$
\langle\nabla_X^{\perp} V,\xi^{\perp}\rangle=\langle\nabla^N_X V,\xi^{\perp}\rangle=-\langle V,\nabla^N_X(\xi-\xi^{\top})\rangle=\langle V,\varphi X\rangle+\langle V,\sigma(X,\xi^{\top})\rangle=0,
$$
for any tangent vector field $X$, where $\xi^{\top}$ is the tangent part of $\xi$.

Next, one obtains
\begin{align*}
\langle\nabla_X^{\perp} V,(\varphi Y)^{\perp}\rangle&=\langle\nabla^N_XV,(\varphi Y)^{\perp}\rangle=-\langle V,\nabla^N_X(\varphi Y-(\varphi Y)^{\top})\rangle\\&=-\langle V,\varphi\nabla_X^NY+\langle X,Y\rangle\xi-\eta(Y)X\rangle+\langle V,\sigma(X,(\varphi Y)^{\top})\rangle\\&=\langle\varphi V,\sigma(X,Y)\rangle\\&=0,
\end{align*}
for any tangent vector fields $X$ and $Y$.

We also get
\begin{align*}
\langle\nabla_X^{\perp} V,\varphi\vec H\rangle&=\langle\nabla^N_XV+A_VX,\varphi\vec H\rangle=-\langle V,\varphi\nabla^N_X\vec H-\eta(H)X\rangle+\langle V,\sigma(X,(\varphi\vec H)^{\top})\rangle\\&=\langle V,\varphi A_{\vec H}X\rangle\\&=0.
\end{align*}

We now choose a local orthonormal frame field $\{E_1,E_2\}$ tangent to the surface such that it diagonalizes $A_{\vec H}$. Since $\vec H$ is not umbilical, from Corollary \ref{lemma_split}, we get that $\{E_1,E_2\}$ also diagonalizes $A_V$, where $V$ is a normal vector field satisfying $V\perp\varphi T\Sigma^2$ and $V\perp\varphi\vec H$. Since $\nabla_X^{\perp}V$ has these properties, it follows that 
\begin{equation}\label{eq:diag}
\langle\nabla^{\perp}_XV,\sigma(E_1,E_2)\rangle=0,
\end{equation}
for any normal vector field $V$ orthogonal to $L$. 

Now, using the fact that $V$ is normal and orthogonal to $L$, the Codazzi equation \eqref{Codazzi} and expression \eqref{eq:curv} of the curvature $R^N$, one obtains
\begin{align*}
\langle\nabla_{E_i}^{\perp}V,\sigma(E_j,E_k)\rangle&=-\langle V,\nabla^{\perp}_{E_i}\sigma(E_j,E_k)\rangle\\&=-\langle V,(\nabla^{\perp}_{E_i}\sigma)(E_j,E_k)+\sigma(\nabla_{E_i}E_j,E_k)+\sigma(E_j,\nabla_{E_i}E_k)\rangle\\&=-\langle V,(\nabla^{\perp}_{E_i}\sigma)(E_j,E_k)\rangle\\&=-\langle V,(\nabla^{\perp}_{E_j}\sigma)(E_i,E_k)+R^N(E_i,E_j)E_k\rangle=-\langle V,\nabla^{\perp}_{E_j}\sigma(E_i,E_k)\rangle\\&=\langle \nabla^{\perp}_{E_j}V,\sigma(E_i,E_k)\rangle,
\end{align*}
that, since $\sigma$ is symmetric, together with \eqref{eq:diag}, shows that $\langle\nabla^{\perp}_{E_i}V,\sigma(E_j,E_k)\rangle=0$ when $i\neq j$ or $i\neq k$ or $j\neq k$.

Next, for $i\neq j$, we have
$$
\langle\nabla^{\perp}_{E_i}V,\sigma(E_i,E_i)\rangle=2\langle\nabla^{\perp}_{E_i}V,\vec H\rangle-\langle\nabla^{\perp}_{E_i}V,\sigma(E_j,E_j)\rangle=-2\langle V,\nabla^{\perp}_{E_i}\vec H\rangle=0,
$$
since $\vec H$ is parallel.

Finally, it only remains to be proved that $\nabla^{\perp}V$ is orthogonal to $(\varphi(\im\sigma))^{\perp}$. This follows from the following computation
\begin{align*}
\langle\nabla^{\perp}_XV,(\varphi\sigma(Y,Z))^{\perp}\rangle&=\langle\nabla^N_XV,(\varphi\sigma(Y,Z))^{\perp}\rangle=-\langle V,\nabla^N_X(\varphi\sigma(Y,Z))^{\perp}\rangle\\&=-\langle V,\nabla^N_X(\varphi\sigma(Y,Z)-(\varphi\sigma(Y,Z)^{\top}))\rangle\\&=-\langle V,\varphi\nabla^N_X\sigma(Y,Z)\rangle+\langle V,\sigma(X,(\varphi\sigma(Y,Z))^{\top})\rangle\\&=\langle\varphi V,\nabla^N_X\sigma(Y,Z)\rangle=-\langle\nabla^N_X\varphi V,\sigma(Y,Z)\rangle\\&=-\langle\nabla^{\perp}_X\varphi V,\sigma(Y,Z)\rangle\\&=0,
\end{align*}
since $\varphi V\perp L$ and, therefore, $\nabla^{\perp}\varphi V\perp\im\sigma$.

Hence, we have just proved that the subbundle $L$ is parallel.
\end{proof}

\begin{proposition}\label{Oku} Let $L$ be the normal subbundle considered in Proposition \ref{lemma_parallel}. Then $T\Sigma^2\oplus L$ is parallel with respect to Okumura's connection $\bar\nabla$ and invariant by $\bar T$ and $\bar R$, i.e., $\bar T(X,Y)\in T\Sigma^2\oplus L$ and $\bar R(X,Y)Z\in T\Sigma^2\oplus L$, for any $X,Y,Z\in T\Sigma^2\oplus L$.
\end{proposition}

\begin{proof} Since $\varphi(T\Sigma^2\oplus L)\subset T\Sigma^2\oplus L$ and $\xi\in T\Sigma^2\oplus L$, it is easy to see that $T\Sigma^2\oplus L$ is invariant by $\bar T$ and $\bar R$. It is also easy to see that $\bar\nabla_XY\in T\Sigma^2\oplus L$, for any vector fields $X$ and $Y$ tangent to $\Sigma^2$.

Next, let us consider a normal vector field $V\in L$. Then 
$$
\bar\nabla_XV=\nabla^N_XV+\mathcal{T}_XV=-A_VX+\nabla^{\perp}_XV+\langle X,\varphi V\rangle\xi-\eta(X)\varphi V+\eta(V)\varphi X,
$$
which means that $\bar\nabla_XV\in T\Sigma^2\oplus L$ if and only if $\nabla_X^{\perp}V\in L$. Since from Proposition \eqref{lemma_parallel} we know that $L$ is parallel, it follows that $T\Sigma^2\oplus L$ is parallel with respect to Okumura's connection $\bar\nabla$.
\end{proof}

It is easy to see that, if $\gamma:I\rightarrow N^{2n+1}(c)$ is a parametrized curve, then $\bar\nabla_{\gamma'}\gamma'=\nabla^N_{\gamma'}\gamma'$, which means that the connections $\bar\nabla$ and $\nabla^N$ have the same geodesics, and therefore, that $\bar\nabla$ is a complete connection. Then, using that $\bar\nabla\bar T=0$, $\bar\nabla\bar R=0$ and Proposition \ref{Oku}, we can apply \cite[Theorem 2]{ET} to prove that our surface lies in an $11$-dimensional totally geodesic submanifold of $N^{2n+1}(c)$. But, since $\varphi(T\Sigma^2\oplus L)\subset L\oplus T\Sigma^2$ and $\xi\in L\oplus T\Sigma^2$, this totally geodesic submanifold actually is a Sasakian space form with the same $\varphi$-sectional curvature $c$ as the ambient space (see \cite{YK}). We conclude with the following proposition. 

\begin{proposition}\label{p_numb} Let $\Sigma^2$ be an isometrically immersed non-minimal pmc surface in a Sasakian space form
$N^{2n+1}(c)$. If its mean curvature vector field is not an umbilical direction, then the surface lies in an $11$-dimensional Sasakian space form $N^{11}(c)$.
\end{proposition}

Now, from Propositions \ref{p_umb} and \ref{p_numb}, we obtain our main result.

\begin{theorem}\label{thm_red} Let $\Sigma^2$ be an isometrically immersed non-minimal pmc surface in a Sasakian space form $N^{2n+1}(c)$ with constant $\varphi$-sectional curvature $c\neq 1$. Then one of the following holds$:$
\begin{enumerate}
\item $\Sigma^2$ is an integral pseudo-umbilical surface and $n\geq 3$$;$ or

\item $\Sigma^2$ is not pseudo-umbilical and lies in a Sasakian space form $N^{11}(c)$.
\end{enumerate}
\end{theorem}

\section{Anti-invariant pmc surfaces}\label{s4}

\subsection{Holomorphic differentials} Let $\Sigma^2$ be an isometrically immersed surface in a Sasakian space form $N^{2n+1}(c)$. Let us consider the following two quadratic forms on $\Sigma^2$:
$$
Q_1(X,Y)=8\langle\sigma(X,Y),\vec H\rangle-(c-1)\eta(X)\eta(Y)
$$
and
$$
Q_2(X,Y)=\langle\varphi X,\vec H\rangle\langle\varphi Y,\vec H\rangle+\eta(X)\eta(Y)-\eta(X)\langle\varphi Y,\vec H\rangle-\eta(Y)\langle\varphi X,\vec H\rangle,
$$
where $\sigma$ is the second fundamental form of the surface and $\vec H$ its mean curvature vector field.

\begin{proposition} If $\Sigma^2$ is an anti-invariant pmc surface in a Sasakian space form $N^{2n+1}(c)$, then the $(2,0)$-parts $Q_1^{(2,0)}$ and $Q_2^{(2,0)}$ of $Q_1$ and $Q_2$, respectively, are holomorphic.
\end{proposition}

\begin{proof} Let us consider isothermal coordinates $(u,v)$ on the surface and then we have $ds^2=\lambda^2(du^2+dv^2)$. Next, define
$z=u+iv$, $\widehat z=u-iv$, $dz=(1/\sqrt{2})(du+idv)$, $d\widehat
z=(1/\sqrt{2})(du-idv)$ and
$$
Z=\frac{1}{\sqrt{2}}\Big(\frac{\partial}{\partial
u}-i\frac{\partial}{\partial v}\Big),\quad \widehat Z=\frac{1}{\sqrt{2}}\Big(\frac{\partial}{\partial
u}+i\frac{\partial}{\partial v}\Big).
$$
It follows that $\langle Z,\widehat Z\rangle=\langle\frac{\partial}{\partial
u},\frac{\partial}{\partial u}\rangle=\langle\frac{\partial}{\partial
v},\frac{\partial}{\partial v}\rangle=\lambda^2$. We mention that this rather unusual notation for conjugation is used only for the reader's convenience.

In the following, we shall compute
$$
\widehat Z(Q_1(Z,Z))=\widehat Z(8\langle\sigma(Z,Z),\vec H\rangle-(c-1)(\eta(Z))^2).
$$

First, we have
\begin{align*}
\widehat Z(\langle\sigma(Z,Z),\vec H\rangle)&=\langle\nabla^N_{\widehat Z}\sigma(Z,Z),\vec H\rangle+\langle\sigma(Z,Z),\nabla^N_{\widehat Z}\vec H\rangle\\
&=\langle\nabla^{\perp}_{\widehat Z}\sigma(Z,Z),\vec H\rangle+\langle\sigma(Z,Z),\nabla^{\perp}_{\widehat Z}\vec H\rangle\\&=\langle(\nabla^{\perp}_{\widehat Z}\sigma)(Z,Z),\vec H\rangle,
\end{align*}
since $\vec H$ is parallel and
$$
(\nabla^{\perp}_{\widehat Z}\sigma)(Z,Z)=\nabla^{\perp}_{\widehat Z}\sigma(Z,Z)-2\sigma(\nabla_{\widehat Z}Z,Z)=\nabla^{\perp}_{\widehat Z}\sigma(Z,Z),
$$
where $\nabla_{\widehat Z}Z=0$, from the definition of $\nabla$.

Next, from the Codazzi equation \eqref{Codazzi}, one obtains
\begin{align}\label{eq:1}
\widehat Z(\langle
\sigma(Z,Z),\vec H\rangle)=&\langle(\nabla^{\perp}_{Z}\sigma)(\widehat Z,Z),\vec H\rangle+\langle (R^N(\widehat Z,Z)Z)^{\perp},\vec H\rangle\\\nonumber &+\langle\sigma(Z,Z),\nabla^{\perp}_{\widehat Z}\vec H\rangle\\\nonumber=&\langle(\nabla^{\perp}_{Z}\sigma)(\widehat Z,Z),\vec H\rangle+\langle R^N(\widehat Z,Z)Z,\vec H\rangle.
\end{align}
We use the expression \eqref{eq:curv} of the curvature $R^N$ and the fact that our surface is anti-invariant, to obtain
\begin{equation}\label{eq:2}
\langle R^N(\widehat Z,Z)Z,\vec H\rangle=\frac{c-1}{4}\langle Z,\widehat Z\rangle\eta(Z)\eta(\vec H).
\end{equation}
Next, it is easy to see that 
\begin{equation}\label{eq:sigma}
\nabla^N_{\widehat Z}Z=\sigma(\widehat Z,Z)=\langle\widehat Z,Z\rangle \vec H
\end{equation}
and that $\nabla_ZZ=\frac{1}{\lambda^2}\langle\nabla_ZZ,\widehat Z\rangle Z$, and then we have
\begin{align}\label{eq:3}
\langle(\nabla^{\perp}_{Z}\sigma)(\widehat Z,Z),\vec H\rangle=&\langle\nabla^N_Z\sigma(\widehat Z,Z),\vec H\rangle-\langle\sigma(\nabla_Z\widehat Z, Z), \vec H\rangle-\langle\sigma(\widehat Z,\nabla_ZZ), \vec H\rangle\\\nonumber
=&\langle\nabla^N_Z(\langle\widehat Z,Z\rangle \vec H),\vec H\rangle-\frac{1}{\lambda^2}\langle\nabla_ZZ,\widehat Z\rangle \langle\sigma(\widehat Z,Z), \vec H\rangle
\\\nonumber=&\langle\nabla^N_Z(\langle\widehat Z,Z\rangle \vec H),\vec H\rangle-\langle\nabla_ZZ,\widehat Z\rangle|\vec H|^2\\\nonumber=&\langle\nabla_Z\widehat Z,Z\rangle|\vec H|^2+\langle\nabla_ZZ,\widehat Z\rangle|\vec H|^2+\langle\widehat Z,Z\rangle\langle\nabla^{\perp}_Z\vec H,\vec H\rangle\\\nonumber &-\langle\nabla_ZZ,\widehat Z\rangle|\vec H|^2\\\nonumber=&0.
\end{align}

Since the ambient space is a Sasakian space form, we use that $\Sigma^2$ is anti-invariant and equation \eqref{eq:sigma} to prove that
\begin{align}\label{eq:term2}
\widehat Z((\eta(Z))^2)&=2\eta(Z)(\langle\nabla^N_{\widehat Z}Z,\xi\rangle+\langle Z,\nabla^N_{\widehat Z}\xi)=2\eta(Z)(\langle Z,\widehat Z\rangle\eta(\vec H)-\langle Z,\varphi\widehat Z\rangle)\\\nonumber &=2\langle Z,\widehat Z\rangle\eta(Z)\eta(\vec H).
\end{align}

From equations \eqref{eq:1}, \eqref{eq:2}, \eqref{eq:3} and \eqref{eq:term2} we conclude that $\widehat Z(Q_1(Z,Z))=0$, i.e., the $(2,0)$-part of $Q_1$ is holomorphic.

In order to prove that also the $(2,0)$-part of $Q_2$ is holomorphic we shall compute $\widehat Z(Q_2(Z,Z))=\widehat Z((\eta(Z)-\langle\varphi Z,\vec H\rangle)^2)$.

First, as we have seen before, we have $\widehat Z(\eta(Z))=\langle Z,\widehat Z\rangle\eta(\vec H)$. 

Next, using the properties of a Sasakian space form, the fact that our surface is anti-invariant and pmc, and equation \eqref{eq:sigma}, we get
\begin{align*}
\widehat Z(\langle\varphi Z,\vec H\rangle)&=\langle\nabla^N_{\widehat Z}\varphi Z,\vec H\rangle+\langle\varphi Z,\nabla^N_{\widehat Z}H\rangle\\&=\langle\varphi\nabla^N_{\widehat Z}Z-\langle Z,\widehat Z\rangle\xi+\eta(Z)\widehat Z,\vec H\rangle-\langle\varphi Z,A_H\widehat Z\rangle\\&=\langle\langle Z,\widehat Z\rangle\varphi\vec H-\langle Z,\widehat Z\rangle\xi,\vec H\rangle\\&=\langle Z,\widehat Z\rangle\eta(\vec H).
\end{align*}

Hence $\widehat Z(Q_2(Z,Z))=0$ and we conclude.
\end{proof}

\subsection{Hopf cylinders} Let us now consider the orbit space $\bar N=N/\xi$ of the Sasakian space form $N^{2n+1}(c)$. Then $\bar N$ is a complex space form with constant holomorphic sectional curvature $c+3$ (see \cite{YK}). The fibration $\pi:N\rightarrow\bar{N}$ is called \textit{the Boothby-Wang fibration}. An example of such a fibration is the well known Hopf fibration $\pi:\mathbb{S}^{2n+1}\rightarrow \mathbb{C}P^{n}$.

Now, let us recall the definition of Frenet curves in a Riemannian manifold. Let $\gamma:I\subset\mathbb{R}\rightarrow M$ be a curve parametrized by
arc-length in a Riemannian manifold $M$. The curve $\gamma$ is called a {\it Frenet curve of osculating order} $r$, $1\leq r\leq 2n$, if there exist $r$ orthonormal vector fields $\{E_1=\gamma',\ldots,E_{r}\}$ along $\gamma$ such that
$$
\nabla^{M}_{E_{1}}E_{1}=\kappa_{1}E_{2},\quad
\nabla^{M}_{E_{1}}E_{i}=-\kappa_{i-1}E_{i-1} + \kappa_{i}E_{i+1},\quad
\nabla^{M}_{E_{1}}E_{r}=-\kappa_{r-1}E_{r-1},
$$
for $i=2,\dots,r-1$, where $\{\kappa_{1},\kappa_{2},\kappa_{3},\ldots,\kappa_{r-1}\}$ are positive
functions on $I$ called the {\it curvatures} of $\gamma$.

A Frenet curve of osculating order $r$ is called a {\it helix of
order $r$} if $\kappa_i=\cst>0$ for $1\leq i\leq r-1$. A helix
of order $2$ is called a {\it circle}, and a helix of order $3$ is
simply called a {\it helix}.

When $\gamma$ is a Frenet curve in a complex space form $\bar N^n(c+3)$, then its {\it complex torsions} are defined by $\tau_{ij}=\langle
E_i, J E_j \rangle$, $1\leq i<j\leq r$, where $(J,\langle,\rangle)$ is the complex structure on $\bar N^n(c+3)$. A helix of order $r$
is called a {\it holomorphic helix of order $r$} if all the complex
torsions are constant. It is easy to see that a circle is always a holomorphic circle (see \cite{MO}).

In order to find examples of anti-invariant pmc surfaces in the Sasakian space form $N^{2n+1}(c)$ we shall first study \textit{Hopf cylinders}, i.e., surfaces $\Sigma^2=\pi^{-1}(\gamma)$, where $\pi:N\rightarrow\bar{N}$ is the Boothby-Wang fibration and $\gamma:I\rightarrow \bar N^n(c+3)$ is a Frenet curve of osculating order $r$ in $\bar N^n(c+3)$. For any vector field $X$ tangent to $\bar N^n(c+3)$ we shall denote by $X^H$ its horizontal lift to $N^{2n+1}(c)$. For the Riemannian metrics on $\bar N^n(c+3)$ and $N^{2n+1}(c)$, we will use the same notation $\langle,\rangle$.

Since $\{E_1^H,\xi\}$ is a local orthonormal frame on $\Sigma^2$ and $E_i^H$, $1<i\leq r$, are normal vector fields, the mean curvature vector $\vec H$ of $\Sigma^2$ is given by
\begin{equation}\label{Hopf}
\vec H=\frac{1}{2}(\sigma(E_1^H,E_1^H)+\sigma(\xi,\xi))=\frac{1}{2}\kappa_1E_2^H,
\end{equation}
where $\kappa_1=\kappa_1\circ\pi$ and we used the first Frenet equation of $\gamma$ and O'Neill's equation for Riemannian submersions in the case of Boothby-Wang fibration, i.e., 
$$
\nabla^N_{X^H}Y^H=(\nabla^{\bar N}_XY)^H-\langle X^H,\varphi Y^H\rangle\xi,
$$
for any vector fields $X$ and $Y$ tangent to $\bar N(c+3)$ (see  \cite{O}). The first consequence of this equation is that $\kappa_1=2|\vec H|$ and then a Hopf cylinder $\Sigma^2=\pi^{-1}(\gamma)$ is minimal in $N^{2n+1}(c)$ if and only if the curve $\gamma$ is a geodesic in $\bar N^n(c+3)$.

In \cite{I}, it is proved that a non-minimal pmc Hopf cylinder $\Sigma^2$ lies in a Sasakian space form $N^3(c)$ of dimension $3$. In this case, since $\varphi E_1^H$ is orthogonal to $\Sigma^2$, it follows that $\vec H=\pm|\vec H|\varphi E_1^H$. Then equation \eqref{Hopf} implies that $\tau_{12}=\langle E_1,JE_2\rangle=\pm 1$, where $J$ is the complex structure on $\bar N^1(c+3)$. 

Next, from the second Frenet equation of $\gamma$, one obtains
\begin{align*}\label{eq:PMC1}
\nabla^N_{E_1^H}\vec H&=\frac{1}{2}\{(\nabla^{\bar N}_{E_1}(\kappa_1E_2))^H-\kappa_1\langle E_1^H,\varphi E_2^H\rangle\xi\}\\&=\frac{1}{2}(\kappa_1'E_2-\kappa_1^2E_1+\kappa_1\kappa_2E_3)^H-\frac{1}{2}\kappa_1\langle E_1^H,\varphi E_2^H\rangle\xi.
\end{align*}

It is easy to verify that $\nabla_{\xi}E_1^H=\nabla_{E_1^H}\xi=0$, where $\nabla$ is the connection on the surface, and then we get that $[\xi,E_1^H]=0$, which means $\nabla^N_{\xi}E_1^H=\nabla^N_{E_1^H}\xi=-\varphi E_1^H$, from where we obtain that $\nabla^N_{\xi}\vec H=-\varphi\vec H=\mp|\vec H|E_1^H$. 

We conclude with the following proposition.

\begin{proposition} A Hopf cylinder $\Sigma^2=\pi^{-1}(\gamma)$ in a Sasakian space form $N^{2n+1}(c)$ has parallel mean curvature vector field $\vec H$ if and only if either
\begin{enumerate}
\item $\gamma$ is a geodesic in the orbit space $\bar N^n(c+3)$, i.e., $\Sigma^2$ is a minimal surface in $N^{2n+1}(c)$; or

\item $\gamma$ is a circle in $\bar N^1(c+3)$ with curvature $\kappa_1=2|\vec H|=\cst$ and complex torsion $\tau_{12}=\pm 1$. In this case $\Sigma^2$ lies in $N^3(c)$.
\end{enumerate}
\end{proposition}

\begin{remark} In \cite{MA}, it is proved that for any positive number $\kappa$ and for any number $\tau$, such that $|\tau|\leq 1$, there exits a circle with curvature $\kappa$ and complex torsion $\tau$ in any complex space form. Therefore, non-minimal pmc Hopf cylinders do exist in $N^3(c)$.
\end{remark}

\begin{remark} It is easy to see that neither $Q_1^{(2,0)}$ nor $Q_2^{(2,0)}$ can vanish on a non-minimal pmc Hopf cylinder, since this would imply that $\vec H$ is orthogonal to $\varphi E_1^H$, which is a contradiction.
\end{remark}

\subsection{Integral surfaces} Let $\Sigma^2$ be an integral non-minimal pmc surface in a $7$-dimensional Sasakian space form $N^7(c)$ such that $Q_1^{(2,0)}=0$ on $\Sigma^2$. It is easy to see that, since our surface is integral, this condition is equivalent with that that $\Sigma^2$ is pseudo-umbilical. From Proposition \ref{p_umb} we know that if $\{E_1,E_2\}$ is an orthonormal frame field tangent to the surface, then $\{E_3=\varphi E_1,E_4=\varphi E_2,E_5=\vec H/|\vec H|,E_6=\varphi E_5,E_7=\xi\}$ is a normal orthonormal frame field, where $\vec H$ is the mean curvature vector field of $\Sigma^2$. We note that also $Q_2^{(2,0)}=0$ on such a surface. 

Since $\Sigma^2$ is integral, it is easy to verify that
\begin{equation}\label{int0}
\langle\sigma(E_i,E_j),\varphi E_k\rangle=\langle\sigma(E_i,E_k),\varphi E_j\rangle,\quad\forall i,j,k=1,2
\end{equation}
and
\begin{equation}\label{int1}
A_7=A_{E_7}=0,
\end{equation}
and, as the surface is pseudo-umbilical, we also have that
\begin{equation}\label{int2}
A_5=A_{E_5}=\left(\begin{array}{cc}|\vec H|&0\\ 0&|\vec H|\end{array}\right).
\end{equation}

In the following, we shall choose the tangent frame field $\{E_1,E_2\}$ such that it diagonalizes $A_3=A_{E_3}$, and then we have
\begin{equation}\label{int3}
A_3=\left(\begin{array}{cc}a&0\\0&-a\end{array}\right),
\end{equation}
where $a:\Sigma^2\rightarrow\mathbb{R}$ is a function on the surface.

Next, using equation \eqref{int0}, we obtain that
$$
\langle\sigma(E_1,E_1),\varphi E_2\rangle=\langle\sigma(E_1,E_2),\varphi E_1\rangle=0,
$$
and 
$$
\langle\sigma(E_1,E_1),\varphi E_1\rangle=-\langle\sigma(E_2,E_2),\varphi E_1\rangle=-\langle\sigma(E_1,E_2),\varphi E_2\rangle=a,
$$
which means that 
\begin{equation}\label{int4}
A_4=A_{E_4}=\left(\begin{array}{cc}0&-a\\-a&0\end{array}\right).
\end{equation}

Since for any tangent vector field $X$ we have $\langle\vec H,\varphi X\rangle=0$ and $\varphi X$ is normal, it follows that
$$
\langle\nabla^{\perp}_Y\vec H,\varphi X\rangle+\langle\nabla^N_Y\varphi X,\vec H\rangle=0,
$$
which gives, using that $\vec H$ is parallel and again that $\varphi\vec H$ is a normal vector field,
$$
\langle\sigma(X,Y),\varphi\vec H\rangle=0,
$$
for any vector field $Y$ tangent to the surface. Therefore, we have
\begin{equation}\label{int5}
A_6=A_{E_6}=0.
\end{equation} 

Next, using the equation of Gauss \eqref{Gauss}, the expression \eqref{eq:curv} of $R^N$ and equations \eqref{int1}--\eqref{int5}, one obtains the Gaussian curvature $K$ of the surface as 
\begin{align}\label{Gcurv}
K&=\langle R(E_1,E_2)E_2,E_1\rangle\\\nonumber &=\langle R^N(E_1,E_2)E_2,E_1\rangle+\langle\sigma(E_1,E_1),\sigma(E_2,E_2)\rangle-|\sigma(E_1,E_2)|^2\\\nonumber &=\frac{c+3}{4}-2a^2+|H|^2.
\end{align}

On the other hand, since $\nabla_{E_i}E_1$ and $\nabla_{E_j}E_2$ are orthogonal, we have
\begin{align}\label{Gcurvdef}
K&=\langle R(E_1,E_2)E_2,E_1\rangle\\\nonumber &=\langle\nabla_{E_1}\nabla_{E_2}E_2,E_1\rangle-\langle\nabla_{E_2}\nabla_{E_1}E_2,E_1\rangle-\langle\nabla_{[E_1,E_2]}E_2,E_1\rangle\\\nonumber &=E_1(\langle\nabla_{E_2}E_2,E_1\rangle)+E_2(\langle\nabla_{E_1}E_1,E_2\rangle)-\langle\nabla_{E_1}E_1,E_2\rangle^2-\langle\nabla_{E_2}E_2,E_1\rangle^2.
\end{align}

\begin{lemma}\label{normal} The following equations hold on $\Sigma^2$$:$
\begin{enumerate}
\item $\nabla^{\perp}_{E_i}\varphi E_i=\langle\nabla_{E_i}E_i,E_j\rangle\varphi E_j+\varphi\vec H+\xi$, with $i\neq j$$;$

\item $\nabla^{\perp}_{E_i}\varphi E_j=\langle\nabla_{E_i}E_j,E_i\rangle\varphi E_i$, if $i\neq j$$;$

\item $\nabla^{\perp}_X\varphi\vec H=-|\vec H|^2\varphi X$ and $\nabla_X^{\perp}\xi=-\varphi X$.
\end{enumerate}
\end{lemma}

\begin{proof} Since $\vec H$ is parallel and umbilical, we have $\langle\nabla^{\perp}_X\varphi E_i,\vec H\rangle=0$ and also
$$
\langle\nabla^{\perp}_X\varphi E_i,\varphi\vec H\rangle =-\langle\varphi E_i,\nabla^N_X\varphi\vec H\rangle=-\langle\varphi E_i,\varphi\nabla^N_X\vec H\rangle=\langle\varphi E_i,\varphi A_{\vec H}X\rangle=|\vec H|^2\langle E_i,X\rangle,
$$
for any vector field $X$ tangent to the surface.

Now, for $i\neq j$, one obtains
$$
\langle\nabla^{\perp}_X\varphi E_i,\varphi E_j\rangle=\langle\nabla^{N}_X\varphi E_i,\varphi E_j\rangle=\langle\varphi\nabla^{N}_X E_i,\varphi E_j\rangle=\langle\nabla^{N}_X E_i, E_j\rangle=\langle\nabla_X E_i, E_j\rangle.
$$

Finally, we again use that $\vec H$ is parallel and umbilical, $\nabla^N_X\xi=-\varphi X$ and $\varphi X$ is normal, to conclude.
\end{proof}

\begin{lemma}\label{deriv_a} The derivatives of the function $a:\Sigma^2\rightarrow\mathbb{R}$ are given by
$$
E_1(a)=3a\langle\nabla_{E_2}E_2,E_1\rangle\quad\textnormal{and}\quad E_2(a)=3a\langle\nabla_{E_1}E_1,E_2\rangle.
$$
\end{lemma}

\begin{proof} Using equations \eqref{int1}--\eqref{int5}, Lemma \ref{normal} and the fact that $\vec H$ is parallel, we can compute
\begin{align*}
(\nabla^{\perp}_{E_1}\sigma)(E_2,E_1)=&\nabla^{\perp}_{E_1}\sigma(E_2,E_1)-\sigma(\nabla_{E_1}E_2,E_1)-\sigma(E_2,\nabla_{E_1}E_1)\\=&-E_1(a)\varphi E_2-a\nabla^{\perp}_{E_1}\varphi E_2-\langle\nabla_{E_1}E_2,E_1\rangle\sigma(E_1,E_1)\\&-\langle\nabla_{E_1}E_1,E_2\rangle\sigma(E_2,E_2)\\=&-E_1(a)\varphi E_2+3a\langle\nabla_{E_1}E_1,E_2\rangle\varphi E_1
\end{align*}
and, in the same way,
\begin{align*}
(\nabla^{\perp}_{E_2}\sigma)(E_1,E_1)=&\nabla^{\perp}_{E_2}\sigma(E_1,E_1)-2\sigma(\nabla_{E_2}E_1,E_1)\\=&E_2(a)\varphi E_1-3a\langle\nabla_{E_2}E_2,E_1\rangle\varphi E_2.
\end{align*}

From the Codazzi equation \eqref{Codazzi} and using \eqref{eq:curv}, one obtains
$$
0=(R^N(E_1,E_2)E_1)^{\perp}=(\nabla^{\perp}_{E_1}\sigma)(E_2,E_1)-(\nabla^{\perp}_{E_2}\sigma)(E_1,E_1),
$$
which leads to the conclusion.
\end{proof}

Now, we are ready to prove the following theorem.

\begin{theorem}\label{thm:int} Let $\Sigma^2$ be an isometrically immersed integral complete non-minimal surface in a Sasakian space form $N^7(c)$ with parallel mean curvature vector field $\vec H$ and with non-negative Gaussian curvature $K$. If $Q_1^{(2,0)}$ vanishes on the surface or, equivalently, if $\Sigma^2$ is pseudo-umbilical, then $|\vec H|^2\geq-(c+3)/4$ and one of the following holds$:$
\begin{enumerate}

\item $\Sigma^2$ is a cmc totally umbilical surface in a space form $M^3((c+3)/4)$, with constant sectional curvature $(c+3)/4$ and of dimension $3$, immersed in $N^7(c)$ as an integral submanifold$;$ or

\item $\Sigma^2$ is flat and it is the standard product $\gamma_1\times\gamma_2$, where $\gamma_1:\mathbb{R}\rightarrow N^7(c)$ is a Legendre helix of osculating order $4$ in $N^7(c)$ with curvatures 
$$
\kappa_1=\sqrt{a^2+|\vec H|^2},\quad \kappa_2=\frac{a\sqrt{1+|\vec H|^2}}{\sqrt{a^2+|\vec H|^2}},\quad\textnormal{and}\quad\kappa_3=\frac{|\vec H|\sqrt{1+|\vec H|^2}}{\sqrt{a^2+|\vec H|^2}},
$$
and $\gamma_2:\mathbb{R}\rightarrow N^7(c)$ is a Legendre circle in $N^7(c)$ with curvature $\kappa=\sqrt{a^2+|\vec H|^2}$, where $a^2=(c+3)/8+|\vec H|^2/2$.
\end{enumerate}
\end{theorem}

\begin{proof} First, we shall prove that $a^2:\Sigma^2\rightarrow\mathbb{R}$ is a subharmonic function. Indeed, after a straightforward computation, using Lemma \ref{deriv_a} and equation \eqref{Gcurvdef}, we have
\begin{align*}
\Delta a^2&=\sum_{i=1}^2(E_i(E_i(a^2))-(\nabla_{E_i}E_i)(a^2))\\&=6a^2(E_1(\langle\nabla_{E_2}E_2,E_1\rangle)+E_2(\langle\nabla_{E_1}E_1,E_2\rangle)+5\langle\nabla_{E_2}E_2,E_1\rangle^2+5\langle\nabla_{E_1}E_1,E_2\rangle^2)\\&=6a^2\{K+6(\langle\nabla_{E_2}E_2,E_1\rangle^2+\langle\nabla_{E_1}E_1,E_2\rangle^2)\}\geq 0.
\end{align*}

Since $K\geq 0$, it follows that $\Sigma^2$ is a parabolic space. Moreover, from equation \eqref{Gcurv}, we get that $2a^2\leq (c+3)/4+|\vec H|^2$, which means that $a^2$ is a bounded subharmonic function on a parabolic space and, therefore, due to a result in \cite{H}, a constant. Then, either $a$ vanishes or the surface is flat and $\nabla E_1=\nabla E_2=0$.

\textbf{Case I: $a=0$.} In this case, $\langle\sigma(X,Y),V\rangle=0$ for any normal vector field $V$ orthogonal to $\vec H$, which means that the normal subbundle
$L=\Span\{\im\sigma\}=\Span\{\vec H\}$ is parallel. Also, it is easy to see that $T\Sigma^2\oplus L$ is invariant by $\bar T$ and $\bar R$, where $\bar T$ and $\bar R$ are the torsion and the curvature, respectively, of Okumura's connection. Moreover, the characteristic vector field $\xi$ is orthogonal to $T\Sigma^2\oplus L$. Therefore, we can use \cite[Theorem 2]{ET} and \cite[Proposition 8.1]{B} to conclude that $\Sigma^2$ lies is a space form $M^3((c+3)/4)$ immersed in $N^7(c)$ as an integral submanifold.

\textbf{Case II: $a\neq 0$.} We have that $K=0$, which gives $a^2=(c+3)/8+|H|^2/2$, and $\nabla E_1=\nabla E_2=0$. Since $E_1$ and $E_2$ are parallel, they determine two distributions which are mutually orthogonal, smooth, involutive and parallel. Therefore, from the de Rham Decomposition Theorem follows, also taking into account that the surface is complete and using its universal cover if necessary, that $\Sigma^2$ is a product $\gamma_1\times\gamma_2$, where $\gamma_i:\mathbb{R}\rightarrow N^7(c)$, $i\in\{1,2\}$, are integral curves of $E_1$ and $E_2$, respectively, parametrized by arc-length, i.e., $\gamma_1'=E_1$ and $\gamma_2'=E_2$ (see \cite{KN}). Moreover, since the surface is integral, the two curves are Legendre curves. In the following, we shall determine their curvatures.

Let $\kappa_i$, $1\leq i<7$, be the curvatures of $\gamma_1$ and $\{X_j^1\}$, $1\leq j<8$, be its Frenet frame field, where $X_1^1=E_1$. From equations \eqref{int1}--\eqref{int5}, we have 
$$
\nabla^N_{E_1}E_1=\sigma(E_1,E_1)=aE_3+|\vec H|E_5
$$
and then, from the first Frenet equation of $\gamma_1$, it follows 
$$
\kappa_1=\sqrt{a^2+|\vec H|^2}\quad\textnormal{and}\quad X_2^1=-\frac{1}{\sqrt{a^2+|\vec H|^2}}(aE_3+|\vec H|E_5).
$$

Next, using Lemma \ref{normal}, we get, after a straightforward computation,
$$
\nabla^N_{E_1}X_2^1=-\kappa_1E_1+\frac{a}{\kappa_1}(\varphi\vec H+\xi),
$$
which shows that 
$$
\kappa_2=\frac{a\sqrt{1+|\vec H|^2}}{\sqrt{a^2+|\vec H|^2}}\quad\textnormal{and}\quad X_3^1=\frac{a}{\sqrt{1+|\vec H|^2}}(|\vec H|E_6+E_7).
$$
It follows that $\nabla^N_{E_1}X_3^1=-\sqrt{1+|\vec H|^2}E_3$
and, from the third Frenet equation,
$$
\kappa_3=\frac{|\vec H|\sqrt{1+|\vec H|^2}}{\sqrt{a^2+|\vec H|^2}}\quad\textnormal{and}\quad X_4^1=-\frac{1}{\sqrt{a^2+|\vec H|^2}}(-|\vec H|E_3+aE_5).
$$
Finally, we obtain $\nabla^N_{E_1}X_4^1=-\kappa_3X_3^1$, which shows that $\gamma_1$ is a helix of osculating order $4$. 

In the case of the curve $\gamma_2$, consider its Frenet frame field $\{X_j^2\}$, $1\leq j<8$, with $X_1^2=E_2$, and, again using equations \eqref{int1}--\eqref{int5}, we have
$$
\nabla^N_{E_2}E_2=\sigma(E_2,E_2)=-aE_3+|\vec H|E_5,
$$
which means that its first curvature is $\kappa=\sqrt{a^2+|\vec H|^2}$ and 
$$
X_2^2=\frac{1}{\sqrt{a^2+|\vec H|^2}}(-aE_3+|\vec H|E_5).
$$
Then, using Lemma \ref{normal}, we obtain $\nabla^N_{E_2}X_2^2=-\kappa E_2$. Therefore, the curve $\gamma_2$ is a circle in $N^7(c)$.
\end{proof}

Using a result in \cite{C} we have the following corollary.

\begin{corollary}\label{theorem3} An integral pmc $2$-sphere in $N^7(c)$, with $\vec H\neq 0$, is a round sphere in a space form $M^3((c+3)/4)$. 
\end{corollary}

\subsection{Anti-invariant surfaces} In the last part of our paper, we shall consider anti-invariant non-minimal pmc surfaces $\Sigma^2$ isometrically immersed in a Sasakian space form $N^{2n+1}(c)$, $c\neq 1$, such that the mean curvature vector field $\vec H$ is not umbilical. We will prove that there are no $2$-spheres with these properties.

First, let $\Sigma^2$ be a surface as above and assume that $Q_1^{(2,0)}$ and $Q_2^{(2,0)}$ vanish on $\Sigma^2$. If the characteristic vector field $\xi$ is orthogonal to $\Sigma^2$ at a point $p$, then, from $Q_1^{(2,0)}=0$, it follows that $\vec H$ is umbilical at $p$. Also, if we assume that $\xi$ is tangent to the surface at a point $p$, then $Q_2^{(2,0)}\neq 0$ at $p$, which is a contradiction. Therefore, since the map $p\in\Sigma^2\rightarrow(A_{\vec H}-|\vec H|^2\id)(p)$ is analytic, it follows that the tangent and normal parts of the characteristic vector field $\xi$ do not vanish on an open dense set. We shall work on this set and then we will extend our results to the whole surface by continuity.

Since the tangent part $\xi^{\top}$ and the normal part $\xi^{\perp}$ of $\xi$ do not vanish, we can choose an orthonormal frame field $\{E_1,E_2\}$ on $\Sigma^2$ such that $E_2=\xi^{\top}/|\xi^{\top}|$. Then we have that $\eta(E_1)=0$ and, from $Q_1^{(2,0)}=0$, it follows that
\begin{equation}\label{lambda}
\langle A_{\vec H}E_1,E_1\rangle=|\vec H|^2-\frac{c-1}{16}(\eta(E_2))^2,\quad\langle A_{\vec H}E_2,E_2\rangle=|\vec H|^2+\frac{c-1}{16}(\eta(E_2))^2
\end{equation}
and $\langle A_{\vec H}E_1,E_2\rangle=0$, which means that $\{E_1,E_2\}$ diagonalizes $A_{\vec H}$.

From $Q_2^{(2,0)}=0$, we also obtain
\begin{equation}\label{q21}
\langle\varphi E_1,\vec H\rangle=0
\end{equation}
and
\begin{equation}\label{q22}
\langle\varphi E_2,\vec H\rangle=\eta(E_2).
\end{equation}

\begin{lemma}\label{computation} The following identities hold on $\Sigma^2$$:$
\begin{align}
\label{lemmacomp1}&\langle\varphi\vec H,\sigma(E_i,E_i)\rangle=\eta(\vec H)-\eta(\sigma(E_i,E_i)),\quad\textnormal{for}\quad i=1,2,\\
\label{lemmacomp2}&\eta(\sigma(E_1,E_1)=-\eta(\nabla_{E_1}E_1),\\
\label{lemmacomp3}&\nabla_{E_2}E_1=\nabla_{E_2}E_2=0,\\
\label{lemmacomp4}&\sigma(E_1,E_2)=-\frac{2}{\eta(E_2)}\varphi E_1,\\
\label{lemmacomp5}&\langle\sigma(E_1,E_1),\varphi E_2\rangle=\eta(E_2)-\frac{2}{\eta(E_2)},\\
\label{lemmacomp6}&\langle\sigma(E_2,E_2),\varphi E_2\rangle=\eta(E_2)+\frac{2}{\eta(E_2)}.
\end{align}
\end{lemma}

\begin{proof} Since $\Sigma^2$ is an anti-invariant pmc surface, from equation \eqref{q21} and using \eqref{q22}, one obtains
\begin{align*} 
0&=E_1(\langle\varphi\vec H,E_1\rangle)=\langle\nabla^N_{E_1}\varphi\vec H,E_1\rangle+\langle\varphi\vec H,\nabla^N_{E_1}E_1\rangle\\\nonumber &=\langle\varphi\nabla^N_{E_1}\vec H,E_1\rangle-\eta(\vec H)+\langle\varphi\vec H,\sigma(E_1,E_1)\rangle-\langle\nabla_{E_1}E_1,E_2\rangle\langle\vec H,\varphi E_2\rangle\\\nonumber&=-\eta(\vec H)+\langle\varphi\vec H,\sigma(E_1,E_1)\rangle-\langle\nabla_{E_1}E_1,E_2\rangle\eta(E_2),
\end{align*}
which means that $\langle\varphi\vec H,\sigma(E_1,E_1)\rangle=\eta(\vec H)+\eta(\nabla_{E_1}E_1)$. On the other hand, from $\eta(E_1)=0$, we easily get that $\eta(\nabla^N_{E_1}E_1)=0$, i.e., $\eta(\nabla_{E_1}E_1)=-\eta(\sigma(E_1,E_1))$. Therefore, we have $\langle\varphi\vec H,\sigma(E_1,E_1)\rangle=\eta(\vec H)-\eta(\sigma(E_1,E_1))$ and $\langle\varphi\vec H,\sigma(E_2,E_2)\rangle=\eta(\vec H)-\eta(\sigma(E_2,E_2))$.

In order to prove the third identity, let us first note that $\varphi E_1$ is orthogonal to $\varphi\vec H$ and, therefore, to its normal part $(\varphi\vec H)^{\perp}$. Then, from the Ricci equation \eqref{Ricci} and \eqref{eq:curv}, we see that $A_{\vec H}$ and $A_{(\varphi\vec H)^{\perp}}$ commute, which means, since $\{E_1,E_2\}$ diagonalizes $A_{\vec H}$ and $\vec H$ is not umbilical, that $\langle\varphi\vec H,\sigma(E_1,E_2)\rangle=0$.  

Now, since $E_2(\langle\varphi\vec H,E_1\rangle)=0$ and $\Sigma^2$ is an anti-invariant pmc surface, using equation \eqref{q22}, one obtains $\langle\varphi\vec H,\sigma(E_1,E_2)\rangle=\langle\nabla_{E_2}E_1,E_2\rangle\eta(E_2)$. Hence, we have $\langle\nabla_{E_2}E_1,E_2\rangle\eta(E_2)=0$, and then $\nabla_{E_2}E_1=\nabla_{E_2}E_2=0$.

Next, let $V$ be a normal vector field orthogonal to $\varphi E_1$. From the Ricci equation \eqref{Ricci} and \eqref{eq:curv}, we get that $A_{\vec H}$ and $A_V$ commute, which implies that $\{E_1,E_2\}$ diagonalizes $A_V$.

Again using \eqref{Ricci} and \eqref{eq:curv}, it follows that
$$
\langle[A_{\vec H},A_{\varphi E_1}]E_1,E_2\rangle=-\langle R^N(E_1,E_2)\vec H,\varphi E_1\rangle=-\frac{c-1}{4}\eta(E_2).
$$
Then, from \eqref{lambda}, one obtains that $\langle\sigma(E_1,E_2),\varphi E_1\rangle=-(2/\eta(E_2))$. Thus, we have $\sigma(E_1,E_2)=-(2/\eta(E_2))\varphi E_1$.

Finally, since our surface is anti-invariant, we get
\begin{align*}
\langle\sigma(E_1,E_1),\varphi E_2\rangle&=\langle\nabla^N_{E_1}E_1,\varphi E_2\rangle=-\langle E_1,\nabla^N_{E_1}\varphi E_2\rangle\\&=-\langle E_1,\varphi\nabla^N_{E_1}E_2-\eta(E_2)E_1\rangle=\langle\varphi E_1,\sigma(E_1,E_2)\rangle+\eta(E_2)\\&=\eta(E_2)-\frac{2}{\eta(E_2)},
\end{align*}
and then, from equation \eqref{q22},
$$
\langle\sigma(E_2,E_2),\varphi E_2\rangle=\langle2\vec H-\sigma(E_1,E_1),\varphi E_2\rangle=\eta(E_2)+\frac{2}{\eta(E_2)},
$$
which ends the proof.
\end{proof}

\begin{lemma}\label{eta} The second fundamental form $\sigma$ of $\Sigma^2$ satisfies
$$
\eta(\sigma(E_1,E_1))=\frac{1}{2}\eta(\vec H)\quad\textnormal{and}\quad\eta(\sigma(E_2,E_2))=\frac{3}{2}\eta(\vec H).
$$
\end{lemma}

\begin{proof} From equation \eqref{lemmacomp1}, we have
\begin{equation}\label{eq:eta1}
\langle\varphi\sigma(E_2,E_2),\sigma(E_1,E_1)\rangle=2\langle\varphi\sigma(E_2,E_2),\vec H\rangle=2\eta(\sigma(E_2,E_2))-2\eta(\vec H).
\end{equation}

Using Lemma \ref{computation}, one obtains
\begin{align}\label{eq:eta2}
\langle\varphi\sigma(E_2,E_2),\sigma(E_1,E_1)\rangle=&\langle\varphi\nabla^N_{E_2}E_2,\nabla^N_{E_1}E_1\rangle-\langle\varphi\nabla^N_{E_2}E_2,\nabla_{E_1}E_1\rangle\\\nonumber=&\langle\nabla^N_{E_2}\varphi E_2,\nabla^N_{E_1}E_1\rangle-\eta(\nabla^N_{E_1}E_1)\\\nonumber&+\eta(E_2)\langle E_2,\nabla^N_{E_1}E_1\rangle\\\nonumber&+\langle\nabla_{E_1}E_1,E_2\rangle\langle\nabla^N_{E_2}E_2,\varphi E_2\rangle\\\nonumber=&\langle\nabla^N_{E_2}\varphi E_2,\nabla^N_{E_1}E_1\rangle-\eta(\sigma(E_1,E_1))\\\nonumber&+\langle\nabla_{E_1}E_1,E_2\rangle\langle\sigma(E_2,E_2),\varphi E_2\rangle\\\nonumber=&\langle\nabla^N_{E_2}\varphi E_2,\nabla^N_{E_1}E_1\rangle-\eta(\sigma(E_1,E_1))+\eta(\nabla_{E_1}E_1)\\\nonumber&+\frac{2\langle\nabla_{E_1}E_1,E_2\rangle}{\eta(E_2)}\\\nonumber=&\langle\nabla^N_{E_2}\varphi E_2,\nabla^N_{E_1}E_1\rangle-2\eta(\sigma(E_1,E_1))\\\nonumber&+\frac{2\langle\nabla_{E_1}E_1,E_2\rangle}{\eta(E_2)}.
\end{align}

Next, from equation \eqref{lemmacomp5}, we have
\begin{equation}\label{aux}
\langle\nabla^N_{E_1}E_1,\varphi E_2\rangle=\langle\sigma(E_1,E_1),\varphi E_2\rangle=\eta(E_2)-\frac{2}{\eta(E_2)},
\end{equation}
which leads to
\begin{equation}\label{eq:eta3}
\langle\nabla^N_{E_1}E_1,\nabla_{E_2}^N\varphi E_2\rangle+\langle\nabla^N_{E_2}\nabla^N_{E_1}E_1,\varphi E_2\rangle=E_2\Big(\eta(E_2)-\frac{2}{\eta(E_2)}\Big).
\end{equation}

Since $\nabla_{E_2}E_1=0$, we get, using \eqref{lemmacomp5} and \eqref{lemmacomp4},
$$
\nabla^N_{E_1}\nabla^N_{E_2}E_1=\nabla^N_{E_1}\sigma(E_1,E_2)=-E_1\Big(\frac{2}{\eta(E_2)}\Big)\varphi E_1-\frac{2}{\eta(E_2)}\nabla^N_{E_1}\varphi E_1, 
$$
and then, since $\eta(\nabla^N_{E_1}E_1)=0$,
$$
\langle\nabla^N_{E_1}\nabla^N_{E_2}E_1,\varphi E_2\rangle=-\frac{2}{\eta(E_2)}\langle\varphi\nabla^N_{E_1} E_1,\varphi E_2\rangle=-\frac{2\langle\nabla_{E_1} E_1, E_2\rangle}{\eta(E_2)}.
$$

It follows that
\begin{align*}
\langle\nabla^N_{E_2}\nabla^N_{E_1}E_1,\varphi E_2\rangle=&\langle R^N(E_2,E_1)E_1,\varphi E_2\rangle+\langle\nabla^N_{E_1}\nabla^N_{E_2}E_1,\varphi E_2\rangle+\langle\nabla^N_{[E_2,E_1]}E_1,\varphi E_2\rangle\\=&\langle R^N(E_2,E_1)E_1,\varphi E_2\rangle-\frac{2\langle\nabla_{E_1} E_1, E_2\rangle}{\eta(E_2)}\\&-\langle\nabla_{E_1}E_2,E_1\rangle\langle\nabla^N_{E_1}E_1,\varphi E_2\rangle.
\end{align*}
From here, using the expression \eqref{eq:curv} of the curvature $R^N$ and equation \eqref{aux}, one obtains that
$$
\langle\nabla^N_{E_2}\nabla^N_{E_1}E_1,\varphi E_2\rangle=\Big(\eta(E_2)-\frac{4}{\eta(E_2)}\Big)\langle\nabla_{E_1} E_1, E_2\rangle.
$$

Replacing in \eqref{eq:eta3} and using \eqref{lemmacomp3}, we have
$$
\langle\nabla^N_{E_1}E_1,\nabla_{E_2}^N\varphi E_2\rangle=\eta(\sigma(E_2,E_2))+\frac{2\eta(\sigma(E_2,E_2))}{(\eta(E_2))^2}-\Big(\eta(E_2)-\frac{4}{\eta(E_2)}\Big)\langle\nabla_{E_1} E_1, E_2\rangle,
$$
and then, from equation \eqref{eq:eta2}, also using 
$$
\eta(\sigma(E_1,E_1))=-\eta(\nabla_{E_1}E_1)=-\langle\nabla_{E_1} E_1, E_2\rangle\eta(E_2),
$$
it follows, after a straightforward computation, that
\begin{align*}
\langle\varphi\sigma(E_2,E_2),\sigma(E_1,E_1)\rangle=&2\eta(\sigma(E_2,E_2))-2\eta(\vec H)\\&+\frac{2}{(\eta(E_2))^2}\eta(\sigma(E_2,E_2)-3\sigma(E_1,E_1)).
\end{align*}

Finally, from equation \eqref{eq:eta1}, one sees that $\eta(\sigma(E_2,E_2))=3\eta(\sigma(E_1,E_1))$, i.e., $\eta(\sigma(E_1,E_1))=(1/2)\eta(\vec H)$ and $\eta(\sigma(E_2,E_2)=(3/2)\eta(\vec H)$.
\end{proof}

Now, we can prove the following proposition.

\begin{proposition} Let $f:\Sigma^2\rightarrow\mathbb{R}$ be a function on $\Sigma^2$ given by 
$$
f=\frac{1}{2}(\eta(\vec H))^2+\frac{1}{3}(1+|\vec H|^2)(\eta(E_2))^2+\frac{c-1}{96}(\eta(E_2))^4.
$$
Then $f$ is a harmonic function. 
\end{proposition}

\begin{proof} First, since $\vec H$ is parallel and $\langle\sigma(E_1,E_2),\vec H\rangle=0$, from equations \eqref{lambda} and \eqref{q21}, we have 
$$
E_1(\eta(\vec H))=\langle\nabla^N_{E_1}\vec H,\xi\rangle+\langle\vec H,\nabla^N_{E_1}\xi\rangle=-\langle A_{\vec H}E_1,\xi\rangle-\langle\vec H,\varphi E_1\rangle=0
$$
and
\begin{align*}
E_2(\eta(\vec H))&=\langle\nabla^N_{E_2}\vec H,\xi\rangle+\langle\vec H,\nabla^N_{E_2}\xi\rangle=-\langle A_{\vec H}E_2,\xi\rangle-\langle\vec H,\varphi E_2\rangle\\&=-\eta(E_2)\Big(1+|\vec H|^2+\frac{c-1}{16}(\eta(E_2))^2\Big).
\end{align*}

Next, since $\nabla_{E_2}E_2=0$, from Lemma \ref{eta}, one obtains
$$
E_2(\eta(E_2))=\langle\nabla_{E_2}^NE_2,\xi\rangle+\langle E_2,\nabla_{E_2}\xi\rangle=\frac{3}{2}\eta(\vec H)-\langle E_2,\varphi E_2\rangle=\frac{3}{2}\eta(\vec H).
$$

From Lemmas \ref{computation} and \ref{eta}, it is easy to see that $\langle\nabla_{E_1}E_1,E_2\rangle=-\eta(\vec H)/(2\eta(E_2))$.

Using all these equations it is now straightforward to compute $\Delta f$, where $\Delta=\trace\nabla^2 =\trace(\nabla\nabla-\nabla_{\nabla})$. We obtain
\begin{align*}
\Delta(\eta(\vec H))^2=&2(\eta(E_2))^2\Big(1+|\vec H|^2+\frac{c-1}{16}(\eta(E_2))^2\Big)\\&-4(\eta(\vec H))^2\Big(1+|\vec H|^2+\frac{5(c-1)}{32}(\eta(E_2))^2\Big),
\end{align*}
$$
\Delta(\eta(E_2))^2=6(\eta(\vec H))^2-3(\eta(E_2))^2\Big(1+|\vec H|^2+\frac{c-1}{16}(\eta(E_2))^2\Big),
$$
$$
\Delta(\eta(E_2))^4=30(\eta(E_2))^2(\eta(\vec H))^2-6(\eta(E_2))^4\Big(1+|\vec H|^2+\frac{c-1}{16}(\eta(E_2))^2\Big)
$$
and, therefore, $\Delta f=0$, which means that the function $f$ is harmonic.
\end{proof}

Now, let us assume that our surface $\Sigma^2$ is complete and has non-negative Gaussian curvature $K$. Then, since $f$ is a bounded harmonic function on a parabolic space, it follows that $f$ is constant, which implies that $\eta(E_2)$ is constant. As we have already seen, $E_2(\eta(E_2))=(3/2)\eta(\vec H)$, and this leads to $\eta(\vec H)=0$. From here, we get $\langle\nabla_{E_1}E_1,E_2\rangle=-\eta(\vec H)/(2\eta(E_2))=0$, i.e., $\nabla_{E_1}E_1=\nabla_{E_1}E_2=0$. Moreover, since 
$$
E_2(\eta(\vec H))=-\eta(E_2)(1+\langle A_{\vec H}E_2,E_2\rangle)=-\eta(E_2)\Big(1+|\vec H|^2+\frac{c-1}{16}(\eta(E_2))^2\Big),
$$ 
one obtains 
\begin{equation}\label{final2}
\langle A_{\vec H}E_2,E_2\rangle)=|\vec H|^2+\frac{c-1}{16}(\eta(E_2))^2=-1,
\end{equation}
that implies that $c<1$.

\begin{theorem}\label{theorem4} There are no anti-invariant complete non-minimal non-pseudo-umbilical pmc surfaces with non-negative Gaussian curvature in a Sasakian space form $N^{2n+1}(c)$, with $c\neq 1$. In particular, there are no anti-invariant non-minimal pmc $2$-spheres in $N^{2n+1}(c)$.
\end{theorem}

\begin{proof} From expression \eqref{eq:curv} of $R^N$, we get that
\begin{equation}\label{final1}
\langle R^N(E_1,E_2)E_1,E_2\rangle=-\frac{c+3}{4}+\frac{c-1}{4}(\eta(E_2))^2.
\end{equation}

Using the fact that $\nabla E_1=\nabla E_2=0$ and \eqref{lemmacomp4}, we have
\begin{align*}
R^N(E_1,E_2)E_1&=\nabla^N_{E_1}\nabla^N_{E_2}E_1-\nabla^N_{E_2}\nabla^N_{E_1}E_1-\nabla^N_{[E_1,E_2]}E_1\\&=\nabla^N_{E_1}\Big(-\frac{2}{\eta(E_2)}\varphi E_1\Big)-\nabla^N_{E_2}\nabla^N_{E_1}E_1\\&=-\frac{2}{\eta(E_2)}(\varphi\sigma(E_1,E_1)+\xi)-\nabla^N_{E_2}\nabla^N_{E_1}E_1,
\end{align*}
and then, also using \eqref{lemmacomp5},
\begin{align*}
\langle R^N(E_1,E_2)E_1,E_2\rangle&=-2+\frac{2}{\eta(E_2)}\langle\sigma(E_1,E_1),\varphi E_2\rangle-\langle\nabla^N_{E_2}\nabla^N_{E_1}E_1,E_2\rangle\\&=-2+\frac{2}{\eta(E_2)}\Big(\eta(E_2)-\frac{2}{\eta(E_2)}\Big)+\langle\sigma(E_1,E_1),\nabla^N_{E_2}E_2\rangle\\&=-\frac{4}{(\eta(E_2))^2}+\langle\sigma(E_1,E_1),\sigma(E_2,E_2)\rangle.
\end{align*}

This, together with \eqref{final1}, gives
$$
\langle\sigma(E_1,E_1),\sigma(E_2,E_2)\rangle=-\frac{c+3}{4}+\frac{c-1}{4}(\eta(E_2))^2+\frac{4}{(\eta(E_2))^2},
$$
and, therefore, from \eqref{final2}, one obtains
\begin{align*}
|\sigma(E_2,E_2)|^2&=2\langle\vec H,\sigma(E_2,E_2)\rangle-\langle\sigma(E_1,E_1),\sigma(E_2,E_2)\rangle\\&=-2+\frac{c+3}{4}-\frac{c-1}{4}(\eta(E_2))^2-\frac{4}{(\eta(E_2))^2}\\&=\frac{1}{4(\eta(E_2))^2}((1-c)(\eta(E_2))^4+(c-5)(\eta(E_2))^2-16).
\end{align*}
Finally, since $c<1$ and $(\eta(E_2))^2<1$, it can be easily verified that $(1-c)(\eta(E_2))^4+(c-5)(\eta(E_2))^2-16<0$, which is a contradiction.
\end{proof}

\end{document}